\documentclass{elsarticle}
\usepackage{amssymb}
\usepackage{amsmath}
\usepackage{amsthm}
\usepackage{setspace}
\usepackage{enumitem}
\usepackage{mathtools}
\usepackage{mathrsfs}
\usepackage{pgfplots}
\usepackage{mathdots}
\usepackage[margin=1.5 in]{geometry}

\usepackage{pgf,tikz}
\usepackage{tikz-cd}
\usepackage{mathrsfs}
\usetikzlibrary{arrows}

\newtheorem{theorem}{Theorem}[section]
\newtheorem{lemma}[theorem]{Lemma}

\theoremstyle{definition}

\newtheorem{example}[theorem]{Example}

\theoremstyle{remark}

\numberwithin{equation}{section}

\begin{document}

\begin{abstract}

Given a ring $R$ with center $Z(R)$, we say a linear map $f:R\rightarrow R$ is commuting if $[f(x),x]=0$ for all $x\in R$. Such a map has a standard form if there exists $\lambda\in R$ and additive $\mu:R\rightarrow Z(R)$ such that $f(x)=\lambda x+\mu(x)$ for all $x\in R$. We characterize the linear commuting maps over the $n\times n$ Heisenberg algebra, showing that such maps need not be of the standard form.
\end{abstract}


\begin{frontmatter}

\title{Commuting maps on the Heisenberg algebra}

\author{Jordan Bounds}
\address{Jordan Bounds, Department of Mathematics, Furman University, Greenville, South Carolina;}
\ead{jordan.bounds@furman.edu}

\author{Ellis Edinkrah}
\address{Ellis Edinkrah, Department of Mathematics, Furman University, Greenville, South Carolina;}
\ead{edinel9@furman.edu}

\end{frontmatter}

\section{Introduction}

	Let $R$ be a ring with center $Z(R)$. A map $f:R\rightarrow R$ is called {\em commuting} if $[f(x),x]=0$ for all $x\in R$ where $[A,B]=AB-BA$ denotes the standard ring commutator. The study of commuting maps dates back to Posner's work in 1957 in which he proved a non-commutative prime ring cannot possess a nonzero commuting derivation \cite{posner57}. Bre{\u s}ar later showed that an additive commuting map $f$ over a simple unital ring $R$ must be of the form \begin{equation}\label{standard form} f(x)=\lambda x+\mu(x)\end{equation} for some $\lambda\in Z(R)$ and additive $\mu:R\rightarrow Z(R)$ \cite{bresar93}. Significant effort has since been applied in investigating commuting maps over various rings and algebras (see, for example,  \cite{bbc, bounds16, cheung01, franca12, luh70, mayne76, mayne84, mayne92, slowikahmed, vukman90, vukman92}) with many concluding that commuting maps under various assumptions tend to possess the so-called {\em standard form} given in Equation \eqref{standard form}. A detailed account of many of these results, along with the developing theory of commuting maps, can be found in the survey paper by Bre{\u s}ar \cite{bresar04}.

	Particular focus has been placed on examining the structure of commuting maps over matrix rings and algebras, yielding interesting results in the settings of upper and strictly upper triangular matrices. In particular, Beidar, Bre{\u s}ar, and Chebotar \cite{bbc} proved that a linear commuting map defined on $T_n=T_n(F)$, the algebra of $n\times n$ upper triangular matrices with entries in a field $F$, must be of the standard form. The first author then showed \cite{bounds16} that linear commuting maps defined on $\mathcal N_n(F)$, the ring of strictly upper triangular matrices with entries in a field $F$ of characteristic zero, are almost of the standard form. This was later generalized to strictly upper triangular matrices with entries in a unital ring by Ko and Liu \cite{koliu23} as follows.

\begin{theorem}[Theorem 1.1 in \cite{koliu23}]\label{koliu}
    Let $R$ be a ring with 1 and suppose $f:\mathcal N_n(R)\rightarrow \mathcal N_n(R)$ is an additive map satisfying $[f(X),X]=0$ for all $X\in N_n(R)$. Then there exists $\lambda\in Z(R)$ and additive maps $\mu:\mathcal N_n(R)\rightarrow Z(\mathcal N_n(R))$, $\nu:\mathcal N_n(R)\rightarrow\Omega$ such that \[f(X)=\lambda X+\mu(X)+\nu(X)\] for all $X\in \mathcal N_n(R)$ where $\nu(X)=e_{1,1}Xae_{2,n-1}+e_{2,n}aXe_{n,n}$.
\end{theorem}

The results regarding commuting maps over strictly upper triangular matrices are quite interesting. As $N_n$ is nilpotent, one might expect the structure of a commuting map to greatly differ from the setting of a prime and semi-prime rings. However, the above theorem indicates that commuting maps on $N_n$ are very close to being of the standard form. In the present paper we investigate the linear commuting maps over another well-known nilpotent algebra. Let $\mathcal H_n=\mathcal H_n(F)$ denote the Heisenberg algebra over a field $F$. That is, $\mathcal H_n$ is the set of all matrices of the form \[\begin{pmatrix}0&x_{1,2}&x_{1,3}&\cdots&x_{1,n}\\ &&&&x_{2,n}\\ &&&&\vdots\\ &&&&x_{n-1,n}\\ &&&&0\end{pmatrix}\] where $x_{i,j}\in F$. $\mathcal H_n$ is a nilpotent sub-algebra of $N_n$ in which the identity $XYZ=0$ holds for every $X,Y,Z\in\mathcal H_n$. Moreover, it is easy to verify that $Z(\mathcal H_n)=\{ae_{1,n}:a\in F\}$.

In this brief note, we examine the linear commuting maps over $\mathcal H_n$. In particular, we show that such maps can differ quite a bit from the standard form. The Heisenberg algebra, and its associated group, are key structures in physics, particularly quantum mechanics, and have connections to numerous fields of mathematical study. Thus, exploring the commuting maps over $\mathcal H_n$ might provide us with some useful tools that can be applied in these varied settings. 

Our primary result is as follows.

\begin{theorem}\label{result}
Let $f:\mathcal H_n\rightarrow\mathcal H_n$ be a linear map satisfying $[f(X),X]=0$ for all $X\in\mathcal H_n$. Then there exists $A\in\mathcal C_n$, $B,C\in\mathcal P_n$, and additive $\zeta:\mathcal H_n\rightarrow Z(\mathcal H_n)$ such that \[f(X)=\{X,A\}+X^\tau B+CX^\tau+\zeta(X)\] for all $X\in\mathcal H_n$.
\end{theorem}
\noindent Here $C_n$ denotes the collection of $n\times n$ matrices with entries in $F$ whose outer boundary (first row and column, $n$-th row and column) contain only zero entries and $P_n$ is the subset of $C_n$ consisting of hollow persymmetric matrices over $F$.

We begin with some preliminary facts and results in Section \ref{prelim}, including several examples to illustrate the interesting structure commuting maps over $\mathcal H_n$ possess. We then conclude with a proof of Theorem \ref{result} in Section \ref{main}


\section{Preliminaries}\label{prelim}

For the remainder of this paper we let $n\ge 3$ be an integer and $F$ a field of characteristic 0. We let $e_{i,j}$ denote the $n\times n$ matrix with entry in position $(i,j)$ equal to 1 and all other entries equal to 0. As stated in the previous section, we use $\mathcal H_n$ to denote the Heisenberg algebra with entries in $F$. We will also use $M_n=M_n(F)$ to denote the collection of $n\times n$ matrices with entries in $F$. Due to the structure of $\mathcal H_n$, it is convenient to work with respect to the anti-diagonal in this setting. As such, we will make use of the anti-commutator $\{A,B\}=AB+BA$. We will also utilize the {\em anti-transpose} $A^{\tau}$, which is the transpose of $A $ with respect to the anti-diagonal. That is, $A^\tau$ is the matrix with $(i,j)$ entry equal to $a_{n-j+1,n-i+1}$. An $n\times n$ matrix $A=(a_{i,j})$ is called {\em persymmetric} if it is symmetric with respect to the anti-diagonal. In other words, $a_{i,j}=a_{n-j+1,n-i+1}$, meaning $A^\tau=A$. Further, we say that $A$ is {\em skew-persymmetric} if $A=-A^\tau$. The following two sets will play key roles in establishing Theorem \ref{result}.

\begin{align*}
C_n&=C_n(F)=\left\{A=\begin{pmatrix}0&0&0\\ 0&A'&0\\ 0&0&0\end{pmatrix}: A'\in M_{n-1}\right\}\\
&\\
P_n&=P_n(F)=\{A\in C_n: A=-A^\tau\ \text{and}\ a_{i,n-i+1}=0\ \text{for all}\ 1\le i\le n\}.
\end{align*}
Colloquially, we refer to $P_n$ as the set of hollow, skew-persymmetric matrices where hollow here refers to having all zero entries along the anti-diagonal.

It is a rather simple task to construct commuting maps which are of the standard form in this setting. For example, fix $A\in \mathcal H_n$ and define $f_A:\mathcal H_n\rightarrow\mathcal H_n$ by $f_A(X)=AX$. Then $f_A(\mathcal H_n)\subset Z(\mathcal H_n)$, making $f_A$ a commuting map for every choice of $A$. As the image of $f_A$ is contained in $Z(\mathcal H_n)$, we have that $f_A$ is trivially of the standard form with $\lambda=0$. It is also possible to generate maps which need not be of the standard form. In the next few examples, we provide several maps which illustrate this point.

\begin{example}[Example 2 in \cite{bounds16}]\label{hn1}
Let $a\in F$ and define $r:\mathcal H_n\rightarrow\mathcal H_n$ by \[r(X)=ax_{1,2}e_{1,n-1}+ax_{n-1,n}e_{2,n}.\] The map $r$ is linear and satisfies \[r(X)X-Xr(X)=ax_{1,2}x_{n-1,n}e_{1,n}-x_{1,2}ax_{n-1,n}e_{1,n}=0,\] making it commuting. We note $r$ is only of the standard form if we take $a=0$ and it aligns with the form present in Theorem \ref{koliu} for all values of $a$ with $\lambda=0$.
\end{example}

\begin{example}\label{hn2}
Let $a_1,a_2,\dots,a_{n-1}\in F$. Define $g:\mathcal H_n\rightarrow \mathcal H_n$ by \[g(X)=\begin{pmatrix} 0&a_1x_{1,2}&a_2x_{1,3}&\cdots &a_{n-2}x_{1,n-1}&a_{n-1}x_{1,n}\\ &&&&&a_1x_{2,n}\\ &&&&&a_2x_{3,n}\\&&&&&\vdots\\ &&&&&a_{n-2}x_{n-1,n}\\ &&&&&0
\end{pmatrix}.\] 
Then $g$ is a linear map satisfying \[g(X)X-Xg(X)=\left[\sum\limits_{i=2}^{n-2}a_{i-1}x_{1,i}x_{i,n}-\sum\limits_{i=2}^{n-2}x_{1,i}a_{i-1}x_{i,n}\right]e_{1,n}=0\] for all $X\in \mathcal H_n$, hence $g$ is commuting. Note that the map $g$ only aligns with the standard form, or the form present in Theorem \ref{koliu}, if $a_1=a_2=\cdots=a_{n-1}$. 
\end{example}

\begin{example}\label{new}
We now consider a more complex example. Let \[B=\begin{pmatrix} 0&0&0&\cdots&0&0&0\\ 0&1&1&\cdots&1&0&0\\ 0&1&0&\cdots &0&-1&0\\ \vdots&\vdots&\vdots&\ddots&\vdots&\vdots&\vdots\\ 0&1&0&\cdots &0&-1&0\\ 0&0&-1&\cdots&-1&-1&0\\0&0&0&\cdots&0&0&0\end{pmatrix}\] and define $h:N_n\rightarrow N_n$ by \[h(X)=[X^\tau,B].\] Through direct calculation we see

\begin{align*}
X^\tau B&= \begin{pmatrix} 0& \sum\limits_{i=3}^{n-1}x_{i,n}&x_{n-1,n}-x_{2,n}&\cdots&x_{n-1,n}-x_{2,n}&-\sum\limits_{i=2}^{n-2}x_{i,n}&0\\ &&&&&&0\\ &&&&&&\vdots\\ &&&&&&0\end{pmatrix}\\
&\\
BX^\tau &=\begin{pmatrix} 0&\cdots&0\\&&\sum\limits_{i=3}^{n-1}x_{1,i}\\ &&x_{1,n-1}-x_{1,2}\\ &&\vdots\\&&x_{1,n-1}-x_{1,2}\\&&-\sum\limits_{i=2}^{n-2}x_{1,i}\\ &&0
\end{pmatrix}.
\end{align*}
Note that \begin{align*}
h(X)X&=\left[ x_{2,n}\sum\limits_{i=3}^{n-1}x_{i,n}+(x_{n-1,n}-x_{2,n})\sum\limits_{i=3}^{n-2}x_{i,n}-x_{n-1,n}\sum\limits_{i=2}^{n-2}x_{i,n}\right]e_{1,n}=0, \\
&\\
Xh(X)&=\left[x_{1,2}\sum\limits_{i=3}^{n-1}x_{1,i}+(x_{1,n-1}-x_{1,2})\sum\limits_{i=3}^{n-2}x_{1,i}-x_{1,n-1}\sum\limits_{i=2}^{n-2}x_{1,i}\right]e_{1,n}=0,
\end{align*} hence $[h(X),X]=0$ and $h$ is commuting. Note that $h$ is of the form described in Theorem \ref{result} with $A=0, \zeta(X)\equiv 0$, and $B=-C$.
\end{example}

We provide one more example of a map of the form given in Theorem \ref{result} where the matrices $B,C$ are distinct and $A$ is nonzero.

\begin{example}\label{newer}
Let $A$ be any matrix in $\mathcal C_n$, $B$ the matrix from Example \ref{new}, and  \[
C=	\begin{pmatrix} 0&0&0&\cdots&0&0&0\\ 0&1&1&\cdots&1&0&0\\ 0&1&1&\iddots&0&-1&0\\ \vdots&\vdots&\iddots&\iddots&\iddots&\vdots&\vdots\\ 0&1&0&\iddots&-1&-1&0\\ 0&0&-1&\cdots&-1&-1&0\\0&0&0&\cdots&0&0&0\end{pmatrix}.	\]
Define $f:\mathcal H_n\rightarrow\mathcal H_n$ by $f(X)=\{X,A\}+X^\tau B+CX^\tau$. Recall $XYZ=0$ for all $X,Y,Z\in\mathcal H_n$. Therefore, \[ [\{X,A\},X]=XAX+AX^2-XAX-X^2A=0\] for all $X\in\mathcal H_n$.

As in Example \ref{new} we have \[X^\tau B= \begin{pmatrix} 0& \sum\limits_{i=3}^{n-1}x_{i,n}&x_{n-1,n}-x_{2,n}&\cdots&x_{n-1,n}-x_{2,n}&-\sum\limits_{i=2}^{n-2}x_{i,n}&0\\ &&&&&&0\\ &&&&&&\vdots\\ &&&&&&0\end{pmatrix}.\] Examining $[X^\tau B,X]$ reveals \begin{align*}
[X^\tau B,X]&=X^\tau BX-0\\
&=\left(x_{2,n}\sum\limits_{i=3}^{n-1}x_{i,n}+(x_{n-1,n}-x_{2,n})\sum\limits_{i=3}^{n-2}x_{i,n}-x_{n-1,n}\sum\limits_{i=2}^{n-2}x_{i,n}\right)e_{1,n}\\
&=(x_{2,n}x_{n-1,n}-x_{n-1,n}x_{2,n})e_{1,n}\\
&=0.
\end{align*}

Through direct calculation we also obtain \[CX^\tau =\begin{pmatrix} &0&&\cdots&0\\ 
&&&&\sum\limits_{i=3}^{n-1}x_{1,i}\\
&&&&\sum\limits_{i=4}^{n-1}x_{1,i}-x_{1,2}\\
&&&&\sum\limits_{i=5}^{n-1}x_{1,i}-\sum\limits_{i=2}^{3}x_{1,i}\\
&&&&\vdots\\
&&&&x_{1,n-1}-\sum\limits_{i=2}^{n-3}x_{1,i}\\
&&&&-\sum\limits_{i=2}^{n-2}x_{1,i}\\
&&&&0
\end{pmatrix}.
\] Examining $[CX^\tau,X]$ we see 
\[[CX^\tau,X]=0-XCX^\tau=\left( \sum\limits_{j=2}^{n-2}\sum\limits_{i=j+1}^{n-1}x_{1,j}x_{1,i}-\sum\limits_{j=3}^{n-1}\sum\limits_{i=2}^{j-1}x_{1,j}x_{1,i} \right)e_{1,n}.\] The double sum $\sum\limits_{j=2}^{n-2}\sum\limits_{i=j+1}^{n-1}x_{1,j}x_{1,i}$ consists of all pairs $x_{1,j}x_{1,i}$ with $j<i$ and each pair appears precisely once. Similarly, the double sum $\sum\limits_{j=3}^{n-1}\sum\limits_{i=2}^{j-1}x_{1,j}x_{1,i}$ consists of all pairs $x_{1,j}x_{1,i}$ with $i<j$ and each pair appearing once. Thus, \[[CX^\tau,X]=\left( \sum\limits_{j=2}^{n-2}\sum\limits_{i=j+1}^{n-1}x_{1,j}x_{1,i}-\sum\limits_{j=3}^{n-1}\sum\limits_{i=2}^{j-1}x_{1,j}x_{1,i} \right)e_{1,n}=0.\]

Therefore, \[[f(X),X]=[\{X,A\},X]+[X^\tau B,X]+[CX^\tau,X]=0,\] making $f$ a commuting map on $\mathcal H_n$.
\end{example}

As the statement of Theorem \ref{result} indicates, maps of the form given in Example \ref{newer} will play a key role in describing the general form of commuting maps over $\mathcal H_n$. To prove this theorem, we first establish a few preliminary results. Suppose $f:\mathcal H_n\rightarrow\mathcal H_n$ is a linear, commuting map. Since $f$ is linear, there exist linear maps $f_{1,i},f_{i,n}:\mathcal H_n\rightarrow F$, $2\le i\le n-1$, such that \[f(X)=\begin{pmatrix}0&f_{1,2}(X)&f_{1,3}(X)&\cdots&f_{1,n}(X)\\ &&&&f_{2,n}(X)\\ &&&&\vdots\\ &&&&f_{n-1,n}(X)\\ &&&&0\end{pmatrix}.\] We begin by exploring the structure of the maps $f_{1,i},f_{i,n}$.

\begin{lemma}\label{l0}
 $f_{1,i}(e_{i,n})=f_{i,n}(e_{1,i})=0$ for each $2\le i\le n-1$.
\end{lemma}
\begin{proof}
As $f$ is commuting we have \begin{align*}0&=f(e_{1,i})e_{1,i}-e_{1,i}f(e_{1,i})\\
&=0-f_{i,n}(e_{1,i}),
\end{align*}
and \begin{align*}0&=f(e_{i,n})e_{i,n}-e_{i,n}f(e_{i,n})\\
&=f_{1,i}(e_{i,n})-0.
\end{align*}
\end{proof}

\begin{lemma}\label{l1}
For each $i,j$ we have \begin{align*} f_{j,n}(e_{1,i})&=-f_{i,n}(e_{1,j}),\\ f_{1,j}(e_{1,i})&=f_{i,n}(e_{j,n}),\\ f_{1,j}(e_{i,n})&=-f_{1,i}(e_{j,n}). \end{align*}
\end{lemma}
\begin{proof}
Replacing $X$ with $e_{1,i}+e_{1,j}$ in $[f(X),X]=0$ yields \begin{align*}0&=f(e_{1,i})e_{1,j}+f(e_{1,j})e_{1,i}-e_{1,j}f(e_{1,i})-e_{1,i}f(e_{1,j})\\
&=0+0-f_{j,n}(e_{1,i})-f_{i,n}(e_{1,j}).
\end{align*}
Thus, \begin{equation}\label{e4}f_{j,n}(e_{1,i})=-f_{i,n}(e_{1,j}).\end{equation}

Next, replacing $X$ with $e_{1,i}+e_{j,n}$ in $[f(X),X]=0$ yields \begin{align*}0&=f(e_{1,i})e_{j,n}+f(e_{j,n})e_{1,i}-e_{j,n}f(e_{1,i})-e_{1,i}f(e_{j,n})\\
&=f_{1,j}(e_{1,i})+0-0-f_{i,n}(e_{j,n}).
\end{align*}
Therefore, \begin{equation}\label{e5}f_{1,j}(e_{1,i})=f_{i,n}(e_{j,n}).\end{equation}

Finally, replacing $X$ with $e_{i,n}+e_{j,n}$ yields \begin{align*}0&=f(e_{i,n})e_{j,n}+f(e_{j,n})e_{i,n}-e_{j,n}f(e_{i,n})-e_{i,n}f(e_{j,n})\\
&=f_{1,j}(e_{i,n})+f_{1,i}(e_{j,n})-0-0.
\end{align*}
Hence, \begin{equation}\label{e6}f_{1,j}(e_{i,n})=-f_{1,i}(e_{j,n}).\end{equation}

\end{proof}

\begin{lemma}\label{l2}
For each $i\in{2,\dots,n-1}$ we have $f_{i,n}(X)=\sum\limits_{j=2}^{n-1}[f_{1,j}(e_{1,i})x_{j,n}-x_{1,j}f_{j,n}(e_{1,i})]$ for all $X\in \mathcal H_n$.
\end{lemma}
\begin{proof}
Replacing $X$ with $X+e_{1,i}$ in $[f(X),X]=0$ and applying the linearity of $f$ yields 
\begin{align*} 0&=f(X)e_{1,i}+f(e_{1,i})X-e_{1,i}f(X)-Xf(e_{1,i})\\
&=0+\sum\limits_{j=2}^{n-1}f_{1,j}(e_{1,i})x_{j,n}-f_{i,n}(X)-\sum\limits_{j=2}^{n-1}x_{1,j}f_{j,n}(e_{1,i}).
\end{align*}
Therefore, \begin{equation}\label{e1}f_{i,n}(X)=\sum\limits_{j=2}^{n-1}[f_{1,j}(e_{1,i})x_{j,n}-x_{1,j}f_{j,n}(e_{1,i})].\end{equation}
\end{proof}

\begin{lemma}\label{l3}
For each $i\in{2,\dots,n-1}$ we have $f_{1,i}(X)=\sum\limits_{j=2}^{n-1}[x_{1,j}f_{j,n}(e_{i,n})-f_{1,j}(e_{i,n})x_{j,n}]$ for all $X\in \mathcal H_n$.
\end{lemma}
\begin{proof}
Replacing $X$ with $X+e_{i,n}$ in $[f(X),X]=0$ and applying the linearity of $f$ yields 
\begin{align*} 0&=f(X)e_{i,n}+f(e_{i,n})X-e_{i,n}f(X)-Xf(e_{i,n})\\
&=f_{1,i}(X)+\sum\limits_{j=2}^{n-1}f_{1,j}(e_{i,n})x_{j,n}-0-\sum\limits_{j=2}^{n-1}x_{1,j}f_{j,n}(e_{i,n}).
\end{align*}
Therefore, \begin{equation}\label{e2}f_{1,i}(X)=\sum\limits_{j=2}^{n-1}[x_{1,j}f_{j,n}(e_{i,n})-f_{1,j}(e_{i,n})x_{j,n}].\end{equation}
\end{proof}


\section{Main Results}\label{main}

We now prove Theorem \ref{result}.

\begin{proof}[Proof of Theorem \ref{result}]

Using Equations \eqref{e1} and \eqref{e2}, we can decompose $f$ as \[f(X)=\alpha(X)-\beta(X)-\gamma(X)+\zeta(X)\] where \begin{align*}
\alpha(X)&=\begin{pmatrix}0&\sum\limits_{j=2}^{n-1}x_{1,j}f_{j,n}(e_{2,n})&\sum\limits_{j=2}^{n-1}x_{1,j}f_{j,n}(e_{3,n})&\cdots&\sum\limits_{j=2}^{n-1}x_{1,j}f_{j,n}(e_{n-1,n})&0 \\
&&&&&\sum\limits_{j=2}^{n-1}f_{1,j}(e_{1,2})x_{j,n}\\
&&&&&\sum\limits_{j=2}^{n-1}f_{1,j}(e_{1,3})x_{j,n}\\
&&&&&\vdots\\
&&&&&\sum\limits_{j=2}^{n-1}f_{1,j}(e_{1,n-1})x_{j,n}\\
&&&&&0\end{pmatrix},\\
\beta(X)&=\begin{pmatrix}0&\sum\limits_{j=2}^{n-1}f_{1,j}(e_{2,n})x_{j,n}&\sum\limits_{j=2}^{n-1}f_{1,j}(e_{3,n})x_{j,n}&\cdots&\sum\limits_{j=2}^{n-1}f_{1,j}(e_{n-1,n})x_{j,n}&0 \\
&&&&&0\\
&&&&&\vdots\\
&&&&&0\end{pmatrix},\\
\gamma(X)&=\begin{pmatrix}0&0&\cdots&0 \\
&&&\sum\limits_{j=2}^{n-1}x_{1,j}f_{j,n}(e_{1,2})\\
&&&\sum\limits_{j=2}^{n-1}x_{1,j}f_{j,n}(e_{1,3})\\
&&&\vdots\\
&&&\sum\limits_{j=2}^{n-1}x_{1,j}f_{j,n}(e_{1,n-1})\\
&&&0\end{pmatrix},
\end{align*}
and $\zeta(X)=f_{1,n}(X)e_{1,n}$. We proceed by examining the structure of these maps.

Let \[A=\begin{pmatrix} 0&0&0&\cdots&0&0\\
0&f_{2,n}(e_{2,n})&f_{2,n}(e_{3,n})&\cdots&f_{2,n}(e_{n-1,n})&0\\
0&f_{3,n}(e_{2,n})&f_{3,n}(e_{3,n})&\cdots&f_{3,n}(e_{n-1,n})&0\\
\vdots&\vdots&\vdots&&\vdots&\vdots\\
0&f_{n-1,n}(e_{2,n})&f_{n-1,n}(e_{3,n})&\cdots &f_{n-1,n}(e_{n-1,n})&0\\
0&0&0&\cdots&0&0 \end{pmatrix}.\] Clearly $A\in\mathcal C_n$. Consider the matrices AX and XA. For $1<i<n$ the $(1,i)$ entry of the matrix $XA$ is given by $\sum\limits_{j=2}^{n-2}x_{1,j}f_{j,n}(e_{i,n})$ and all other entries will be equal to zero. For the matrix $AX$, we have the $(i,n)$ entry, $1<i<n$, is given by $\sum\limits_{j=2}^{n-2}f_{i,n}(e_{j,n})x_{j,n}$ with all other entries equal to zero. Applying Equation \eqref{e5} we get $\sum\limits_{j=2}^{n-2}f_{i,n}(e_{j,n})x_{j,n}=\sum\limits_{j=2}^{n-2}f_{1,j}(e_{1,i})x_{j,n}$. Therefore, 

 \[\{X,A\}=XA+AX=\alpha(X).\]

Next, let \[B=\begin{pmatrix} 0&0&0&\cdots&0&0\\
0&f_{1,n-1}(e_{2,n})&f_{1,n-1}(e_{3,n})&\cdots&f_{1,n-1}(e_{n-1,n})&0\\
0&f_{1,n-2}(e_{2,n})&f_{1,n-2}(e_{3,n})&\cdots&f_{1,n-2}(e_{n-1,n})&0\\
\vdots&\vdots&\vdots&&\vdots&\vdots\\
0&f_{1,2}(e_{2,n})&f_{1,2}(e_{3,n})&\cdots &f_{1,2}(e_{n-1,n})&0\\
0&0&0&\cdots&0&0 \end{pmatrix}.\] Then the matrix $X^\tau B$ has $(1,i)$ entry for $1<i<n$ equal to $\sum\limits_{j=2}^{n-1}x_{j,n}f_{1,j}(e_{i,n})$ and all other entries are equal to zero, thus \[X^\tau B=\beta(X).\]

Clearly $B\in C_n$. By Equation \eqref{e6} we have \[b_{i,j}=f_{1,n-i+1}(e_{j,n})=-f_{1,j}(e_{n-i+1,n})=-b_{n-j+1,n-i+1}\] for all $(i,j)$. As $f_{1,i}(e_{i,n})=0$ for all $i$ by Lemma \ref{l0} and $b_{i,n-i+1}=f_{1,n-i+1}(e_{n-i+1,n})$ for $1<i<n$, it follows that $B\in\mathcal P_n$.

Finally, take \[C=\begin{pmatrix} 0&0&0&\cdots&0&0\\
0&f_{n-1,n}(e_{1,2})&f_{n-2,n}(e_{1,2})&\cdots&f_{2,n}(e_{1,2})&0\\
0&f_{n-1,n}(e_{1,3})&f_{n-2,n}(e_{1,3})&\cdots&f_{2,n}(e_{1,3})&0\\
\vdots&\vdots&\vdots&&\vdots&\vdots\\
0&f_{n-1,n}(e_{1,n-1})&f_{n-2,n}(e_{1,n-1})&\cdots &f_{2,n}(e_{1,n-1})&0\\
0&0&0&\cdots&0&0 \end{pmatrix}.\] The matrix $CX^\tau$ has $(i,n)$ entry for $1<i<n$ equal to $\sum\limits_{j=2}^{n-2}f_{j,n}(e_{1,i})x_{1,j}$ and all other entries are equal to zero. Therefore \[CX^\tau=\gamma(X).\] By Equation \eqref{e4} we have 
\[c_{i,j}=f_{n-j+1,n}(e_{1,i})=-f_{i,n}(e_{1,n-j+1})=-c_{n-j+1,n-i+1}\] for all $(i,j)$. Again applying Lemma \ref{l0} we get $c_{i,n-i+1}=f_{i,n}(e_{1,i})=0$ for all $1<i<n$. Therefore, $C\in\mathcal P_n$.

Thus, we have \[f(X)=\alpha(X)-\beta(X)-\gamma(X)+\zeta(X)=\{X,A\}-X^\tau B+CX^\tau+\zeta(X)\] where $A\in\mathcal C_n$, $B,C\in\mathcal P_n$ and $\zeta:\mathcal H_n\rightarrow Z(\mathcal H_n)$. To see that $\zeta$ is additive, note \begin{align*}\zeta(X+Y)&=f(X+Y)-\{X+Y,A\}+(X+Y)^\tau B+C(X+Y)^\tau\\
&=f(X)+f(Y)-\{X,A\}-\{Y,A\}+X^\tau B+Y^\tau B +CX^\tau+CY^\tau\\
&=\zeta(X)+\zeta(Y).
\end{align*}

\end{proof}


\section*{Acknowledgement}
This material is based upon work supported by the National Science Foundation under Award No. 2316995.

\bibliographystyle{amsplain}

\begin{thebibliography}{widestlabel}
		
 \bibitem{bbc} K.I. Beidar, M. Bre{\u s}ar, M. A. Chebotar, Functional identities on upper triangular matrix algebras, {\em J. Math. Sci.} (New York) \textbf{102} (2000), 4557-4565.
  
 \bibitem{bounds16} J. Bounds, Commuting maps over the ring of strictly upper triangular matrices, {\em Lin. Alg. Appl.}, {\bf 507} (2016), 132--136.
  
 \bibitem{bresar93} M. Bre{\u s}ar, Centralizing mappings and derivations in prime rings, {\em J. algebra}, {\bf 156}, 385--394 (1993).
 

 
 
 
 \bibitem{bresar04} M. Bre{\u s}ar, Commuting maps: a survey, {\em Taiwanese J. Math.} \textbf{8} (2004), 361-397.

 \bibitem{cheung01} W.-S. Cheung, Commuting maps of triangular algebras, {\em J. London Math. Soc.} \textbf{63} (2001), 117-127.
 
 
 \bibitem{franca12} W. Franca, Commuting maps on some subsets of matrices that are not closed under addition, {\em Linear Algebra Appl.} \textbf{437} (2012), 388-391.
 
 
  
 \bibitem{koliu23} S.W. Ko and C.K. Liu, Commuting maps on strictly upper triangular matrix rings, {\em Operators and Matrices} (2023), {\bf 17:} 1023--1036.
 
 
 
  \bibitem{luh70} J. Luh, A note on commuting automorphisms of rings, {\em Amer. Math. Monthly} \textbf{77} (1970), 61-62.

  \bibitem{mayne76} J. Mayne, Centralizing automorphisms of prime rings, {\em Canad. Math. Bull.} \textbf{19} (1976), 113-115.
 
 \bibitem{mayne84} J. Mayne, Centralizing mappings of prime rings, {\em Canad. Math. Bull.} \textbf{27} (1984), 122-126.
 
 \bibitem{mayne92} J. Mayne, Centralizing automorphisms of Lie ideals in prime rings, {\em Canad. Math. Bull.} \textbf{35} (1992), 510-514.
 
 \bibitem{posner57}E. C. Posner, Derivations in prime rings, {\em Proc. Amer. Math. Soc.} \textbf{8} (1957), 1093-1100.
 
  \bibitem{slowikahmed} R. S\l{}owik and D.A.H. Ahmed, $m$-commuting maps on triangular and strictly triangular infinite matrices, {\em Electon. J. Linear Algebra} {\bf 37} (2021), 247--255.

 
 \bibitem{vukman90} J. Vukman, Commuting and centralizing mappings in prime rings, {\em Proc. Amer. Math. Soc.} \textbf{109} (1990), 47-52.
 
  \bibitem{vukman92} J. Vukman, On derivations in prime rings and Banach algebras, {\em Proc. Amer. Math. Soc.} \textbf{116} (1992), 877-884.

 
 
\end{thebibliography}

\end{document}